\documentclass[10pt]{article}
\usepackage{amsthm}
\usepackage{tikz-cd}
\usepackage{shuffle}
\usepackage{times}
\usepackage{amsmath}
\usepackage{amssymb}
\usepackage{hyperref}
\usepackage[nottoc,numbib]{tocbibind}
\newtheorem{thm}{Theorem}[section]
\newtheorem{rmk}[thm]{Remark}
\newtheorem{emp}[thm]{Example}
\newtheorem{lem}[thm]{Lemma}

\newtheorem{defn}[thm]{Definition}
\newtheorem{prop}[thm]{Proposition}
\newtheorem{cor}[thm]{Corollary}

\newtheorem{noco}[thm]{Notation and Convention}
\begin{document}
\title{On iterated universal extensions and Nori's fundamental group of nilpotent bundles}
\author{Xiaodong Yi\footnote{yixd97@outlook.com}}
\date{}
\maketitle
\begin{abstract}
Let $k$ be a field of characteristic $0$, $X$ be a geometrically connected, smooth and proper variety over $k$ and $x\in X(k)$ be a base point. Using the notion of iterated universal extensions, we show that Nori's fundamental group $\pi_{1}^{N}(X,x)$ of nilpotent bundles is uniquely determined by the coherent cohomology groups $\mathrm{H}^{i}(X)=\mathrm{H}^{i}(X,\mathcal{O}_{X})$, $i=1,2$, and the cup product  $\cup: \mathrm{H}^{1}(X)\otimes\mathrm{H}^{1}(X) \rightarrow \mathrm{H}^{2}(X)$. This can be seen as an analogue of a classical fact on the de Rham fundamental group of compact Kähler manifolds. As a byproduct, we also determine low degree group cohomology of the trivial representation of $\pi_{1}^{N}(X,x)$, notably in degree $2$.
\end{abstract}
 \par
\noindent \textbf{Mathematics Subject Classification (MSC 2020):} 14F35, 14L15, 18M25, 20G05.\par
\section{Introduction}
Let $k$ be a field. For a  geometrically connected and proper variety over $k$, with a base point $x\in X(k)$, Nori introduces in \cite{nori1982fundamental} the category of nilpotent bundles, i.e, iterated extensions of $\mathcal{O}_{X}$ and proves that this category is neutral Tannakian. We use $\pi_{1}^{N}(X,x)$ to denote the corresponding Tannakian group scheme. \par
To motivate our results, we start with a ``digression'' about the de Rham fundamental group, whose formulation is of the same spirit as that of Nori's fundamental group. Topologically, for a connected complex manifold $X$ with a base point $x\in X$, the de Rham fundamental group $\pi_{1}^{dR}(X,x)$ is the pro-unipotent completion over $\mathbb{C}$ of the topological fundamental group $\pi_{1}^{top}(X,x)$. Algebraically, let $k$ be a field of characteristic $0$ and $X$ be a geometrically connected and smooth variety over $k$ with a base point $x\in X(k)$. The de Rham fundamental group $\pi_{1}^{dR}(X,x)$  is the Tannakian group scheme associated to the category of nilpotent integrable connections. Here nilpotent integrable connections are defined as iterated extensions of $(\mathcal{O}_{X},d)$. By Riemann-Hilbert correspondence, for a smooth connected complex variety, the topologically defined $\pi_{1}^{dR}(X^{an},x)$ is the same as the algebraically defined $\pi_{1}^{dR}(X,x)$.\par
 A remarkable fact (Corollary 2, Section 6, \cite{deligne1975real}) is that, for a compact Kähler manifold $X$, $\pi_{1}^{dR}(X,x)$ is uniquely determined by $\mathrm{H}_{dR}^{i}(X)=\mathrm{H}_{dR}^{i}(X,(\mathcal{O}_{X},d))$, $i=1,2$, and the cup product  \[\cup: \mathrm{H}_{dR}^{1}(X)\otimes\mathrm{H}_{dR}^{1}(X) \rightarrow \mathrm{H}_{dR}^{2}(X).\]
In fact, we can be more precise about the Hopf algebra $\mathcal{O}(\pi_{1}^{dR}(X,x))$: the Hopf algebra $\mathcal{O}(\pi_{1}^{dR}(X,x))$ can be obtained by applying a bar construction to the differential graded algebra $\Omega^{*}(X)$ of smooth forms, as showed in \cite{chen1977iterated}. The formality of compact Kähler manifolds says that $\Omega^{*}(X)$ is weakly equivalent to the differential graded algebra $\mathrm{H}^{*}_{dR}(X)$, so in the bar construction we can use $\mathrm{H}^{*}_{dR}(X)$ instead, leading to an explicit ``quadratic presentation'' of $\mathcal{O}(\pi_{1}^{dR}(X,x))$. We refer to \cite{amoros1996fundamental}\cite{gil2017multiple} for expositions of the de Rham theory.  \par
Our result says that a similar statement holds for Nori's fundamental group of nilpotent bundles. 
\begin{thm}[Theorem \ref{quadric}]
\label{intro_quadric}
Let $k$ be a field of characteristic $0$, $X$ be a geometrically connected, smooth and proper variety over $k$ and $x\in X(k)$ be a base point. Then Nori's fundamental group $\pi_{1}^{N}(X,x)$ is uniquely determined by the coherent cohomology groups $\mathrm{H}^{i}(X)=\mathrm{H}^{i}(X,\mathcal{O}_{X})$, $i=1,2$, and the cup product 
\[\cup: \mathrm{H}^{1}(X)\otimes\mathrm{H}^{1}(X) \rightarrow \mathrm{H}^{2}(X).\]
\end{thm}
We will proceed as follows. In section $2$ we will include some discussion around the notion of an iterated unversal extension, at a categorical level. The notion of an iterated universal extension is not new, as for the special case of nilpotnet bundles it already appears in \cite{nori1982fundamental}. A more careful analysis of the iterated universal extensions of étale $\mathbb{Q}_{\ell}$-local systems and of integrable connections on curves is given in \cite{10.1215/00127094-1444296}\cite{10.1215/00127094-3146817}. The argument in \cite{10.1215/00127094-1444296} can be copied verbatim to show that, for a $k$-linear tensor abelian category $\mathcal{C}$ (under mild assumption), we can read off the Hopf algebra of the Tannakian group scheme of nilpotent objects, from an iterated universal extension. As a consequence, we will show in Corollary \ref{non_qua} that the Hopf algebra of the Tannakian group scheme of nilpotent objects can be approximated by an explicit ``quadratically presented'' Hopf algebra. The approximation is not canonical, but morphisms between associated graded pieces with respect to the conilpotency filtration are canonical. We remark that this approximation is not tight in general. \par
In section $3$, we will apply the general construction from section $2$ to study nilpotent bundles and nilpotent connections. As sketched above, the formality of compact Kähler manifolds ensures that the approximation given in Corollary \ref{non_qua} is an isomorphism in the case of nilpotent connections on a smooth projective variety. By comparing the iterated universal extensions of bundles and of connections, and using Hodge theory, we will be able to show that the approximation given in Corollary \ref{non_qua} is an isomorphism in the case of nilpotent bundles as well.  \par
As a consequence of our presentation for the Hopf algebra of $\pi_{1}^{N}(X,x)$, in section 4 we determine low degree group cohomology of the trivial representation of $\pi_{1}^{N}(X,x)$.
\begin{thm}[Theorem \ref{cohomology}]
Notations as above. The low degree group cohomology of the trivial representation of $\pi_{1}^{N}(X,x)$ is given by 
\[\mathrm{H}^{i}=
\begin{cases}
k, &\text{ for } i=0, \\
\mathrm{H}^{1}(X) &\text{ for } i=1, \\
\mathrm{Im} (\cup:\mathrm{H}^{1}(X)^{\otimes 2}\rightarrow \mathrm{H}^{2}(X)) &\text{ for } i=2.\\
\end{cases}\]
\end{thm}
We explain the meaning of the group cohomology in degree $2$, which is less obvious. The dimension of this cohomology is, roughly speaking, the number of relations needed, to depict $\pi_{1}^{N}(X,x)$ as a quotient of a free pro-unipotent group scheme. We refer to Remark \ref{relation} for the precise statement. In fact, by the same argument, the statement is valid for de Rham fundamental groups as well (after replacing coherent cohomology by de Rham cohomology). \par
We end the introduction by giving several remarks on existing results and literature. First, although we focus on the characteristic  $0$ case, Nori's fundamental group scheme of nilpotent bundles is well-defined in characteristic $p>0$ as well. In positive characteristic, this group scheme is of totally different nature, i.e, it is pro-finite (Proposition 3, Chapter \uppercase\expandafter{\romannumeral4}, Part \uppercase\expandafter{\romannumeral2}, \cite{nori1982fundamental}). Second, a unipotent homotopy theory is recently carried out in \cite{mondal2026unipotent}, where the technique is different from ours (without Tannakian theory nor Hodge theory). Third, our treatment for an iterated universal extension mainly grows from \cite{nori1982fundamental}\cite{10.1215/00127094-1444296}\cite{10.1215/00127094-3146817}, but to ask for a ``quadratically presented'' Hopf algebra is motivated by the study of Kähler groups, i.e, groups which can be realized as the topological fundamental group of a compact Kähler manifold. Indeed, a necessary condition for a finitely presented group to be Kähler is that its associated Malcev Lie algebra over $\mathbb{R}$ is ``quadratically presented''. See \cite{amoros1996fundamental} for an exposition. The ``quadratically presentedness'' there for pro-nilpotent Lie algebras seems equivalent to the ``quadratically presentedness'' of conilpotent commutative Hopf algebras in our sense, under the equivalence between pro-nilpotent Lie algebras and pro-unipotent affine group schemes (and therefore their Hopf algebras). Finally, the notion of a universal extension also appears in \cite{esnault2006gauss}\cite{d2022universal}, with a ``dualized'' formulation and with rather different purposes from ours.  
\subsection*{\normalsize Convention and Notation}
\begin{itemize}
\item
We use $k$ to denote a field of characteristic $0$. We use $\bar{k}$ to denote the algebraic closure of $k$. By a variety over $k$, we mean an integral scheme which is separated and of finite type over $k$. 
\item We only consider integrable connections on a smooth variety, so the adjective ``integrable'' is occasionally omitted. 
\item We use $\mathrm{AffGrp}_{k}$ to denote the category of affine group schemes over $k$. An affine group scheme is said to be an algebraic group, if its coordinate ring is a finitely generated $k$-algebra. We use $\mathrm{LieAlg}_{k}$ to denote the category of Lie algebras over $k$. We use $\mathrm{Vect}_{k}$ to denote the category of finite dimensional vector spaces over $k$.
\item We use $\cup$ to denote the cup product in the sense of derived functors. Occasionally we use Yoneda cup products as well, and clarification is always made if there is possible confusion.
\end{itemize}

\section{The iterated universal extension and the Hopf algebra of a pro-unipotent affine group scheme}

\begin{noco} Let $\mathcal{C}$ be a $k$-linear abelian tensor category, with unit $1_{\mathcal{C}}$. For any $\mathcal{E},\mathcal{F}\in \mathcal{C}$, we assume $\mathrm{Hom}(\mathcal{E},\mathcal{F})$ and the Yoneda $1$-extension $\mathrm{Ext}^{1}(\mathcal{E},\mathcal{F})$ to be finite dimensional over $k$. We also write $\mathrm{Ext}^{0}$ for $\mathrm{Hom}$. We assume $\mathrm{End}(1_{\mathcal{C}})=k$. 
\end{noco}

By a free object in $\mathcal{C}$, we mean an object of the form $V\otimes 1_{\mathcal{C}}$, where $V$ is a finite dimensional $k$-vector space. We recall the notion of a universal extension, along with some basic properties.
\begin{lem}
\label{univ_ext}

For any object $\mathcal{E}\in\mathcal{C}$, there is a unique extension, called the universal extension of $\mathcal{E}$
\begin{equation}
\label{univ_def}
\begin{tikzcd}
&0 \ar[r] &\mathrm{Ext}^{1}(\mathcal{E},1_{\mathcal{C}})^{\vee}\otimes1_{\mathcal{C}} \ar[r] &U(\mathcal{E}) \ar[r] &\mathcal{E} \ar[r] &0
\end{tikzcd}
\end{equation}
such that
\begin{enumerate}
\item It is universal in the sense that for any extension of $\mathcal{E}$ by a free object
\[\begin{tikzcd}
&0 \ar[r] & V\otimes1_{\mathcal{C}} \ar[r] &\mathcal{F} \ar[r] &\mathcal{E} \ar[r] &0, 
\end{tikzcd}
\]
there is a unique $k$-linear morphism $g:\mathrm{Ext}^{1}(\mathcal{E},1_{\mathcal{C}})^{\vee}\rightarrow V$, and a commutative diagram
\[\begin{tikzcd}[row sep=large]
&0 \ar[r] &\mathrm{Ext}^{1}(\mathcal{E},1_{\mathcal{C}})^{\vee}\otimes1_{\mathcal{C}}  \ar[r]\ar[d,"g\otimes id"]  &U(\mathcal{E}) \ar[r]\ar[d] &\mathcal{E} \ar[r]\ar[d,equal] &0 \\
&0 \ar[r] &V\otimes1_{\mathcal{C}} \ar[r] &\mathcal{F} \ar[r] &\mathcal{E} \ar[r] &0.
\end{tikzcd}
\] We remark that the morphism $U(\mathcal{E})\rightarrow \mathcal{F}$ in the middle may not be unique. 
\item 
Apply $\mathrm{Hom}(\cdot, 1_{\mathcal{C}})$ to \eqref{univ_def}, we obtain a long exact sequence
\[
\begin{tikzcd}[column sep=small]
&0 \ar[r] &\mathrm{Hom}(\mathcal{E},1_{\mathcal{C}}) \ar[r] &\mathrm{Hom}(U(\mathcal{E}),1_{\mathcal{C}}) \ar[r] &\mathrm{Ext}^{1}(\mathcal{E},1_{\mathcal{C}}) \ar[r,"\delta"] &\mathrm{Ext}^{1}(\mathcal{E},1_{\mathcal{C}}), 
\end{tikzcd}
\]
and the connecting morphism $\delta$ is the identity.
\end{enumerate}
\end{lem}
\begin{proof}
See Page 100 \cite{nori1982fundamental} for the universal extension of a vector bundle on a proper variety. The argument can be copied verbatim for a general category $\mathcal{C}$.
\end{proof}
\begin{defn}An object $\mathcal{E}\in\mathcal{C}$ is said to be nilpotent of depth $\leq n$, if there exists a filtration 
\[0=\mathcal{E}_{0}\subset\mathcal{E}_{1}\subset...\subset \mathcal{E}_{n}=\mathcal{E},\]
with $\mathcal{E}_{i+1}/\mathcal{E}_{i}$ being free for any $i$.
\end{defn}
\begin{lem}[Proposition 1.2.1 \cite{shiho2000crystalline}]
The full subcategory $\mathcal{N}(\mathcal{C})$ of $\mathcal{C}$ consisting of nilpotent objects is a $k$-linear abelian rigid tensor category. 
\end{lem}
\begin{noco}
We fix a fiber functor $\omega: \mathcal{N}(\mathcal{C})\rightarrow \mathrm{Vect}_{k}$, i.e, a $k$-linear tensor functor which is faithful and exact. The category $\mathcal{N}(\mathcal{C})$ is neutral Tannakian, and we use $\pi_{1}^{nilp}(\mathcal{C},\omega)$ to denote the corresponding Tannakian group scheme over $k$. 
\end{noco}
A pointed object in $\mathcal{N}(\mathcal{C})$ is a pair $(\mathcal{E},e)$, with $\mathcal{E}\in \mathcal{N}(\mathcal{C})$ and $e\in\omega(\mathcal{E})$. A morphism between pointed objects $\alpha: (\mathcal{E},e)\rightarrow (\mathcal{F},f)$ is a morphism $\alpha:\mathcal{E}\rightarrow \mathcal{F}$ such that $\omega(\alpha)$ sends $e$ to $f$.

\begin{defn}
We define an iterated universal extension of pointed nilpotent objects to be a projective system of pointed nilpotent objects
\[\begin{tikzcd}
&{...}\ar[r] &(\mathcal{E}_{n},e_{n}) \ar[r,"\pi_{n-1}"] &{...} \ar[r,"\pi_{1}"] &(\mathcal{E}_{1},e_{1}),
\end{tikzcd}
\]
where $\mathcal{E}_{1}$ is $1_{\mathcal{C}}$ and $e_{1}$ is an arbitrary non-zero element in $\omega(1_{\mathcal{C}})$. When $(\mathcal{E}_{n-1},e_{n-1})$ is defined, we define $\pi_{n-1}:\mathcal{E}_{n}\rightarrow\mathcal{E}_{n-1}$ to be the universal extension $U(\mathcal{E}_{n-1})\rightarrow\mathcal{E}_{n-1}$ as in Lemma \ref{univ_def}. We take $e_{n}$ to be an arbitrary element in $\omega(\mathcal{E}_{n})$, which is mapped to $e_{n-1}$ by $\omega(\pi_{n})$.
\end{defn}
\begin{prop}
\label{ite_univ}
Fix an iterated universal extension of pointed nilpotent objects
\[\begin{tikzcd}
&{...}\ar[r] &(\mathcal{E}_{n},e_{n}) \ar[r,"\pi_{n-1}"] &{...} \ar[r,"\pi_{1}"] &(\mathcal{E}_{1},e_{1}).
\end{tikzcd}
\]
For each $n$, $\mathcal{E}_{n}$ is nilpotent of depth $\leq n$ with the following universal property:
for any pointed nilpotent object $(\mathcal{F},f)$ of depth $\leq n$, there is a unique morphism $(\mathcal{E}_{n},e_{n})\rightarrow (\mathcal{F},f)$ between pointed objects. 
\end{prop}
\begin{proof}
We proceed by induction. The case $n=1$ is fine. \par
Now we verify the universal property of $(\mathcal{E}_{n},e_{n})$. Let $(\mathcal{F},f)$ be a pointed nilpotent object of depth $\leq n$. There is a quotient $\mathcal{F}'$ of $\mathcal{F}$ which is nilpotent of depth $\leq n-1$ such that the kernel of $\mathcal{F}\twoheadrightarrow \mathcal{F}'$ is free, say $V\otimes 1_{\mathcal{C}}$. Let $f'$ be the image of $f$ along the map $\omega(\mathcal{F})\twoheadrightarrow \omega(\mathcal{F}')$ so $\mathcal{F}'$ is naturally pointed. Use the universal property of $(\mathcal{E}_{n-1},e_{n-1})$, there is a unique morphism $\alpha:(\mathcal{E}_{n-1},e_{n-1})\rightarrow (\mathcal{F}',f')$. We obtain the following commutative diagram
\[
\begin{tikzcd}[row sep=large]
&0 \ar[r] &V\otimes1_{\mathcal{C}} \ar[r]\ar[d,equal] &\tilde{\mathcal{F}} \ar[r]\ar[d] &\mathcal{E}_{n-1} \ar[d,"\alpha"]\ar[r] &0\\
&0 \ar[r] &V\otimes1_{\mathcal{C}} \ar[r]  &\mathcal{F} \ar[r] &\mathcal{F}' \ar[r] &0
\end{tikzcd}
\]
where the right square is a pull-back diagram and $\tilde{\mathcal{F}}$ is naturally pointed by some $\tilde{f}$ (which is sent to $e_{n-1}$ along the horizontal map, and to $f$ along the vertical map). Since $\mathcal{E}_{n}$ is the universal extension of $\mathcal{E}_{n-1}$, we have  a unique $g:\mathrm{Ext}^{1}(\mathcal{E}_{n-1},1_{\mathcal{C}})^{\vee}\rightarrow V$ and a commutative diagram
\[
\begin{tikzcd}[row sep=large]
&0 \ar[r] &\mathrm{Ext}^{1}(\mathcal{E}_{n-1},1_{\mathcal{C}})^{\vee}\otimes 1_{\mathcal{C}} \ar[r]\ar[d,"g\otimes id"] &\mathcal{E}_{n} \ar[r,"\pi_{n-1}"] \ar[d] &\mathcal{E}_{n-1} \ar[d,equal]\ar[r] &0\\
&0 \ar[r] &V\otimes1_{\mathcal{C}} \ar[r]\ar[d,equal] &\tilde{\mathcal{F}} \ar[r]\ar[d] &\mathcal{E}_{n-1} \ar[d,"\alpha"]\ar[r] &0\\
&0 \ar[r] &V\otimes1_{\mathcal{C}} \ar[r]  &\mathcal{F} \ar[r] &\mathcal{F}' \ar[r] &0
\end{tikzcd}
\] Now the composition of two vertical maps in the middle gives $\beta:\mathcal{E}_{n}\rightarrow \mathcal{F}$.  We show that we can modify $\beta$ so that $\omega(\beta)$ sends $e_{n}$ to $f$. Indeed, the difference $f-\omega(\beta)(e_{n})$ in $\omega(\mathcal{F})$ has zero image in $\omega(\mathcal{F}')$, so it lies in $V\otimes\omega(1_{\mathcal{C}})$. Using the universal property of $(\mathcal{E}_{n-1},e_{n-1})$ from the induction hypothesis, there is a unique morphism $\gamma: (\mathcal{E}_{n-1},e_{n-1})\rightarrow (V\otimes 1_{\mathcal{C}},f-\omega(\beta)(e_{n}))$. It suffices to replace $\beta$ by $\beta+\gamma\circ \pi_{n-1}$. \par
Now for any pointed object $(\mathcal{F},f)$ of depth $\leq n$, we have constructed a morphism $\beta: (\mathcal{E}_{n},e_{n})\rightarrow (\mathcal{F},f)$. We prove the uniqueness of $\beta$. We apply induction on the depth of $\mathcal{F}$. This is clear if $\mathcal{F}$ is of depth $\leq 1$, since it is then a free object and we have a sequence of isomorphisms
\[\begin{tikzcd}[column sep=small] &\mathrm{Hom}(\mathcal{E}_{1},1_{\mathcal{C}})\ar[r,"\sim"] &{...} \ar[r,"\sim"] &\mathrm{Hom}(\mathcal{E}_{n-1},1_{\mathcal{C}})\ar[r,"\sim"] &\mathrm{Hom}(\mathcal{E}_{n},1_{\mathcal{C}}). \end{tikzcd}\]
For an arbitrary $\mathcal{F}$ we consider an exact sequence
\[\begin{tikzcd}[column sep=small, row sep=large]
0\ar[r] &V\otimes1_{\mathcal{C}} \ar[r] &\mathcal{F} \ar[r] &\mathcal{F}' \ar[r] &0
\end{tikzcd}
\]
with $\mathcal{F}'$ of smaller depth. If there are morphisms $\beta_{1},\beta_{2}:(\mathcal{E}_{n},e_{n})\rightarrow (\mathcal{F},f)$, the induction hypothesis says that $\beta_{1}-\beta_{2}$ is a morphism $(\mathcal{E}_{n},e_{n})\rightarrow (V\otimes 1_{\mathcal{C}},0)$, and therefore $\beta_{1}-\beta_{2}=0$.  
\end{proof}
\begin{lem}
\label{depth}
If $\mathcal{E}$ is nilpotent of depth $\leq m$ and $\mathcal{F}$ is nilpotent of depth $\leq n$, $\mathcal{E}\otimes\mathcal{F}$ is nilpotent of depth $\leq m+n$. 
\end{lem}
\begin{proof}
We prove by induction on $n+m$. There are $\mathcal{E}'\subset\mathcal{E}$ of depth $\leq m-1$ and $\mathcal{F}'\subset \mathcal{F}$ of depth $\leq n-1$, such that $\mathcal{E}/\mathcal{E}'$ and $\mathcal{F}/\mathcal{F}'$ are free. Now by induction hypothesis, $\mathcal{E}'\otimes\mathcal{F}$ and $\mathcal{E}\otimes\mathcal{F}'$ are of depth $\leq n+m-1$. Therefore the summation $\mathcal{E}'\otimes\mathcal{F}+\mathcal{E}\otimes\mathcal{F}'$ in $\mathcal{E}\otimes\mathcal{F}$ has depth $\leq n+m-1$. We conclude by noting $\mathcal{E}\otimes\mathcal{G}/(\mathcal{E}'\otimes\mathcal{F}+\mathcal{E}\otimes\mathcal{F}')\cong \mathcal{E}/\mathcal{E}'\otimes \mathcal{F}/\mathcal{F}'$, which is free.
\end{proof}

\begin{prop}
\label{hopf_nil}
Fix an iterated universal extension \[\begin{tikzcd}
&{...}\ar[r] &(\mathcal{E}_{n},e_{n}) \ar[r,"\pi_{n-1}"] &{...} \ar[r,"\pi_{1}"] &(\mathcal{E}_{1},e_{1}).
\end{tikzcd}
\]The following data defines the Hopf algebra $\mathcal{O}(\pi_{1}^{nilp}(\mathcal{C},\omega))$.
\begin{enumerate}
\item The underlying vector space is \[\varinjlim_{n}\omega(\mathcal{E}^{\vee}_{n}).\]
\item We have $k\rightarrow \omega (\mathcal{E}_{n})$ sending $1$ to $e_{n}$. Dualizing and passing to colimit, we have the counit 
\[\epsilon: \varinjlim_{n} \omega(\mathcal{E}^{\vee}_{n}) \rightarrow k.\]
\item We have $\mathcal{E}_{n}\rightarrow \mathcal{E}_{1}$ sending $e_{n}$ to $e_{1}$. Therefore we have a map $\omega(\mathcal{E}_{n})\rightarrow \omega(\mathcal{E}_{1})\xrightarrow{\sim}k$, where the last map sends $e_{1}$ to $1\in k$. Dualizing and passing to colimit we have the unit 
\[\mu: k \rightarrow \varinjlim_{n} \omega(\mathcal{E}_{n}^{\vee}).\]
\item
By Lemma \ref{depth}, there is a unique morphism $\mathcal{E}_{n+m}\rightarrow \mathcal{E}_{m}\otimes\mathcal{E}_{n}$  sending $e_{n+m}$ to $e_{m}\otimes e_{n}$. Dualizing, applying the functor $\omega$ and passing to colimit, we have the product \[m:\varinjlim_{n} \omega(\mathcal{E}^{\vee}_{n}) \otimes \varinjlim_{n} \omega(\mathcal{E}^{\vee}_{n}) \rightarrow \varinjlim_{n} \omega(\mathcal{E}^{\vee}_{n}).\]
\item For any $s\in\omega(\mathcal{E}_{n})$ there is a unique morphism $f_{s}: (\mathcal{E}_{n},e_{n})\rightarrow (\mathcal{E}_{n},s)$. Consider the morphism $\omega(\mathcal{E}_{n})\otimes \omega(\mathcal{E}_{n})\rightarrow \omega(\mathcal{E}_{n})$, defined as $(s,t)\mapsto\omega(f_{s})(t)$.  Dualizing and passing to colimit we have the coproduct
\[\Delta: \varinjlim_{n} \omega(\mathcal{E}^{\vee}_{n})\rightarrow \varinjlim_{n}\omega(\mathcal{E}^{\vee}_{n})\otimes \varinjlim_{n} \omega(\mathcal{E}^{\vee}_{n}).\] 
\item According to the construction of the counit and the unit, $k$ is a direct summand of $\omega(\mathcal{E}_{n})$. There exists a unique morphism $\mathcal{E}_{n}\rightarrow\mathcal{E}_{n}$ sending $e_{n}$ to $1\in k\subset \omega(\mathcal{E}_{n})$. Dualizing, applying the functor $\omega$ and passing to colimit, we get the antipode
\[S:\varinjlim_{n} \omega(\mathcal{E}^{\vee}_{n})\rightarrow\varinjlim_{n} \omega(\mathcal{E}^{\vee}_{n}) .\]
\end{enumerate}
\end{prop}
\begin{proof}
This is essentially Proposition 2.9 \cite{10.1215/00127094-1444296}. The argument there can be copied verbatim. 
\end{proof}
\begin{rmk}
\label{intersection_univ}
Consider a $k$-linear tensor functor $F: \mathcal{C}'\rightarrow \mathcal{C}''$, and assume the fiber functors for the categories $\mathcal{N}(\mathcal{C}')$ and $\mathcal{N}(\mathcal{C}'')$ of nilpotent objects are chosen compatibly in the sense that $\omega'=\omega''\circ F$. 
We fix an iterated universal extension \[
\begin{tikzcd}
&...\ar[r,"\pi_{n}'"]&(\mathcal{E}_{n}',e_{n}')\ar[r,"\pi_{n-1}'"] &...\ar[r,"\pi_{1}'"]&(\mathcal{E}_{1}',e_{1}')
\end{tikzcd}\]
resp. 
\[
\begin{tikzcd}
&...\ar[r,"\pi_{n}''"]&(\mathcal{E}_{n}'',e_{n}'')\ar[r,"\pi_{n-1}''"] &...\ar[r,"\pi_{1}''"]&(\mathcal{E}_{1}'',e_{1}'')
\end{tikzcd}\]
for $\mathcal{C}'$ (resp. $\mathcal{C}''$). There is a unique morphism $\mathcal{E}_{n}''\rightarrow F(\mathcal{E}_{n}')$ sending $e_{n}''$ to $e_{n}'$, and it obviously induces a morphism $\omega'((\mathcal{E}'_{n})^{\vee}) \rightarrow \omega''((\mathcal{E}''_{n})^{\vee})$, which fits into a commutative square 
\[
\begin{tikzcd}[row sep=large]
&\omega'((\mathcal{E}'_{n})^{\vee}) \ar[r]\ar[d,hook] &\omega''((\mathcal{E}''_{n})^{\vee}) \ar[d,hook] \\
&\mathcal{O}(\pi_{1}^{nilp}(\mathcal{C}',\omega')) \ar[r]&\mathcal{O}(\pi_{1}^{nilp}(\mathcal{C}'',\omega'')).
\end{tikzcd}
\]
In case that the induced morphism $\pi_{1}^{nilp}(\mathcal{C}'',\omega'')\rightarrow \pi_{1}^{nilp}(\mathcal{C}',\omega')$ is faithfully flat, or equivalently the morphism 
$\mathcal{O}(\pi_{1}^{nilp}(\mathcal{C}',\omega')) \rightarrow\mathcal{O}(\pi_{1}^{nilp}(\mathcal{C}'',\omega'')) $ is an inclusion, we have 
\[\omega'((\mathcal{E}_{n}')^{\vee}) =\mathcal{O}(\pi_{1}^{nilp}(\mathcal{C}',\omega')) \cap \omega''((\mathcal{E}_{n}'')^{\vee}),\]
which can be checked by the universal property. 
\end{rmk}\par
We take the chance to describe an another Hopf algebra. 
\begin{lem}
\label{free_uni}
Let $V$ be a $d$-dimensional vector space over $k$. The following data defines a Hopf algebra $\mathcal{H}(V)$:
\begin{enumerate}
\item The underlying vector space is \[\bigoplus _{n\geq 0}V^{\otimes n}.\]
\item The unit and the counit are defined by natural morphisms $k=V^{\otimes 0}\hookrightarrow \bigoplus _{n\geq 0}V^{\otimes n}$ and $\bigoplus _{n\geq 0}V^{\otimes n}\twoheadrightarrow V^{\otimes 0}=k$ respectively.  
\item The product 
\[ m: \mathcal{H}(V)\bigotimes \mathcal{H}(V)\rightarrow \mathcal{H}(V)\] is the shuffle product. Namely, for $x_{1}\otimes...\otimes x_{m}\in V^{\otimes m}$ and $x_{m+1}\otimes...\otimes x_{n+m}\in V^{\otimes n}$, 
we define 
\[m(x_{1}\otimes...\otimes x_{m}\bigotimes x_{m+1}\otimes...\otimes x_{m+n})=\sum_{\sigma\in\shuffle(m,n)}x_{\sigma^{-1}(1)}\otimes...\otimes x_{\sigma^{-1}(n+m)}\] in $V^{\otimes(n+m)}$,
where $\shuffle (m,n)$ is the set of  permutations $\sigma$ of $\{1,...,n+m\}$, such that $\sigma(1)<...<\sigma(m)$ and $\sigma(m+1)<...<\sigma(n+m)$. \par
\item The coproduct \[\Delta: \mathcal{H}(V)\rightarrow \mathcal{H}(V)\bigotimes \mathcal{H}(V)\] is defined as:
\[\Delta(x_{1}\otimes...\otimes x_{n})=\sum_{i=0}^{n}x_{1}\otimes...\otimes x_{i}\bigotimes x_{i+1}\otimes...\otimes x_{n},\]
where the first and the last terms in the summation should be understood as $1\bigotimes x_{1}\otimes...\otimes x_{n}$ and $x_{1}\otimes...\otimes x_{n}\bigotimes 1$ respectively. 
\item The antipode $S$ is defined as 
\[S(x_{1}\otimes...\otimes x_{n})=x_{n}\otimes...\otimes x_{1}.\]
\end{enumerate} \par
The Hopf algebra $\mathcal{H}(V)$ is the Hopf algebra of a pro-unipotent affine group scheme $\mathcal{G}(V)$, which is the pro-unipotent completion over $k$ of the group freely generated by $d$ elements. 
\end{lem}
\begin{proof}
The dual of $\mathcal{H}(V)$ is a completed Hopf algebra, which is treated as Proposition 8.4.1 \cite{fresse2017homotopy}. Our description is obtained by dualizing once again. 
\end{proof}
\begin{emp}[the argument is essentially from \cite{BAO2025103646}]
\label{cohomology_free}
Let $V$ be a vector space over $k$ of dimension $d$. Let $\mathcal{C}$ be the category $\mathrm{Rep}^{f}\mathcal{G}(V)$ of finite dimensional $k$-linear representations of $\mathcal{G}(V)$ and $\omega$ be the forgetful functor to $\mathrm{Vect}_{k}$. We use $\mathcal{D}$ to denote the category of all $k$-linear representations, and we define group cohomology by injective resolutions in $\mathcal{D}$ and then taking $\mathcal{G}(V)$-invariants. We claim for all representations $\mathcal{F}\in\mathcal{D}$, we have $\mathrm{H}^{2}(\mathcal{G}(V),\mathcal{F})=0$. A consequence is that, in an iterated universal extension, we have $\mathrm{H}^{2}(\mathcal{G}(V),\mathcal{E}^{\vee}_{n})=0$ for all $n\geq 1$, and we have canonical isomorphisms $\mathrm{H}^{1}(\mathcal{G}(V),\mathcal{E}^{\vee}_{n})\cong V^{\otimes n}$, $n\geq 1$, from the long exact sequences
\[\begin{tikzcd}[column sep=small ]
&0 \ar[r] &\mathrm{H}^{1}(\mathcal{G}(V),\mathcal{E}^{\vee}_{n}) \ar[r] &\mathrm{H}^{1}(\mathcal{G}(V),\mathcal{E}^{\vee}_{n-1})\otimes\mathrm{H}^{1}(\mathcal{G}(V),1_{\mathcal{C}}) \ar[r] &\mathrm{H}^{2}(\mathcal{G}(V),\mathcal{E}^{\vee}_{n-1}) \ar[r] &...
\end{tikzcd}\] associated to the defining sequences for $\mathcal{E}_{n}$, $n\geq 2$.  \par
By extension of scalars we reduce to the case $k=\mathbb{C}$. Each object $\mathcal{F}\in\mathcal{D}$ can be written as $\varinjlim_{i}\,\mathcal{F}_{i}$ with $\mathcal{F}_{i}\in\mathcal{C}$. By Riemann-Hilbert correspondence each $\mathcal{F}_{i}$ corresponds to an integrable connection $\tilde{\mathcal{F}}_{i}$ on $X=\mathbb{P}^{1}\setminus\{p_{1},...,p_{d}\}$. We write $\tilde{\mathcal{F}}=\varinjlim_{i}\,\tilde{\mathcal{F}_{i}}$ and define its de Rham cohomology to be $\mathrm{H}^{*}_{dR}(X,\tilde{\mathcal{F}})=\varinjlim_{i}\,\mathrm{H}^{*}_{dR}(X,\tilde{\mathcal{F}_{i}})$. There are natural morphisms between cohomology groups $\mathrm{H}^{*}(\mathcal{G}(V),\mathcal{F})\rightarrow \mathrm{H}^{*}_{dR}(X,\tilde{\mathcal{F}})$, which is obviously an isomorphism in degree $1$. Consider a short exact sequence
\[\begin{tikzcd}
&0 \ar[r] &\mathcal{F} \ar[r] &\mathcal{I} \ar[r] &\mathcal{I}/\mathcal{F} \ar[r] &0,
\end{tikzcd}\] with $\mathcal{I}$ being an injective object in $\mathcal{D}$. We conclude from the following commutative diagram with exact horizontal sequences
\[\begin{tikzcd}[row sep=large]
&0\ar[d,equal] & & &0\ar[d,equal] \\
&\mathrm{H}^{1}(\mathcal{G}(V),\mathcal{I}) \ar[r]\ar[d] &\mathrm{H}^{1}(\mathcal{G}(V),\mathcal{I}/\mathcal{F}) \ar[r]\ar[d] &\mathrm{H}^{2}(\mathcal{G}(V),\mathcal{F}) \ar[r]\ar[d] &\mathrm{H}^{2}(\mathcal{G}(V),\mathcal{I}) \ar[d] \\
&\mathrm{H}_{dR}^{1}(X,\tilde{\mathcal{I}}) \ar[r] &\mathrm{H}_{dR}^{1}(X,\tilde{\mathcal{I}}/\tilde{\mathcal{F}}) \ar[r] &\mathrm{H}_{dR}^{2}(X,\tilde{\mathcal{F}}) \ar[r]\ar[d,equal] &\mathrm{H}_{dR}^{2}(X,\tilde{\mathcal{I}})\ar[d,equal] \\ 
& & &0 &0.
\end{tikzcd}
\]
We remind that the first and the second vertical morphisms are isomorphisms, the vanishing of group cohomology is because $\mathcal{I}$ is injective, and the vanishing of de Rham cohomology is because $X$ is an affine curve. \par
In fact, we can use the same argument to show that $\mathrm{H}^{i}(\mathcal{G}(V),\mathcal{F})=0$ for any $i\geq 2$ and $\mathcal{F}\in\mathcal{D}$. \par
\end{emp}
We relate an iterated universal extension to a more familiar notion, i.e, the conilpotency filtration of a coaugmented coalgebra. 
\begin{defn}
Let $\mathcal{H}=(\mathcal{H},\mu,\epsilon,\Delta)$ be a coaugmented coalgebra. The dual $(\mathcal{H}^{\vee},\epsilon^{\vee},\mu^{\vee},\Delta^{\vee})$ is an augmented completed algebra, with an ideal $\mathcal{I}^{\vee}$ defined as the kernel of $\mu^{\vee}$. There is the conilpotency filtration
\[\mathcal{C}_{0}\mathcal{H}\subset \mathcal{C}_{1}\mathcal{H}\subset...\] defined by setting $\mathcal{C}_{n}\mathcal{H}=\mathrm{Ann}_{\mathcal{H}}(\mathcal{I}^{\vee})^{n+1}$, where $\mathrm{Ann}_{\mathcal{H}}(\mathcal{I}^{\vee})^{n+1}$ is the subspace of $\mathcal{H}$ killed by  $(\mathcal{I}^{\vee})^{n+1}$ under the natural pairing $\langle\cdot,\cdot \rangle: \mathcal{H}\otimes\mathcal{H}^{\vee}\rightarrow k$. \par
\end{defn}
\begin{emp}
\label{conil_free}
We consider the Hopf algebra in Proposition \ref{free_uni}. For each $n\geq 0$, we have $\mathcal{C}_{n}\mathcal{H}(V)=\bigoplus_{i\leq n}V^{\otimes i}$.
\end{emp}
\begin{rmk}
\label{intersection_conilpotency}
Suppose we have a morphism between coaugmented coalgebras $\mathcal{H}'\rightarrow \mathcal{H}''$, which is an inclusion between the underlying vector spaces. We have 
\[
\mathcal{C}_{n}\mathcal{H}'=\mathcal{C}_{n}\mathcal{H}''\cap \mathcal{H}'.
\]
\end{rmk} \par
\begin{lem}
\label{conil}
Let $\mathcal{H}=(\mathcal{H},\mu,\epsilon,m,\Delta,S)$ be a commutative Hopf algebra, and $\mathcal{F}$ be a representation of the corresponding affine group scheme, which is nilpotent of depth $\leq n$. Then $\mathcal{F}$ is a $\mathcal{H}$-comodule and the comodule structure factors as
\[\Delta_{\mathcal{F}}:
\mathcal{F}\rightarrow \mathcal{F}\otimes\mathcal{C}_{n-1}\mathcal{H} \hookrightarrow \mathcal{F}\otimes\mathcal{H}.
\]
\end{lem}
\begin{proof}
By condition, there exists a filtration $0=\mathcal{F}_{0}\subset\mathcal{F}_{1}\subset\mathcal{F}_{2}\subset...\subset \mathcal{F}_{n}=\mathcal{F}$, such that for any $i$ and any $f_{i}\in\mathcal{F}_{i}$, we have
\[\Delta_{\mathcal{F}}(f_{i})\in f_{i}\otimes 1+ \mathcal{F}_{i-1}\otimes\mathcal{H}.\]
We claim \[\Delta_{\mathcal{F}}(\mathcal{F}_{i})\subset \mathcal{F}_{i}\otimes \mathcal{C}_{i-1}\mathcal{H}\] for any $i$. We proceed by induction on $i$. This is clear for $i=1$. To prove 
\[\Delta_{\mathcal{F}}(\mathcal{F}_{i+1})\subset \mathcal{F}_{i+1}\otimes \mathcal{C}_{i}\mathcal{H},\] it suffices to prove for any $f_{i+1}\in\mathcal{F}_{i+1}$, $g\in (\mathcal{I}^{\vee})^{i}$ and $h\in \mathcal{I}^{\vee}$, we have $\langle\Delta_{\mathcal{F}}(f_{i+1}),id\otimes gh\rangle=0$. We have \[\langle\Delta_{\mathcal{F}}(f_{i+1}),id\otimes gh\rangle=\langle( id\otimes\Delta)\circ\Delta_{\mathcal{F}}(f_{i+1}),id\otimes g\otimes h\rangle=\langle( \Delta_{\mathcal{F}}\otimes id)\circ\Delta_{\mathcal{F}}(f_{i+1}),id\otimes g\otimes h\rangle,\]
which is zero since $\Delta_{\mathcal{F}}(f_{i+1})\in f_{i+1}\otimes 1+\mathcal{F}_{i}\otimes \mathcal{H}$ and by induction hypothesis we have $\Delta_{\mathcal{F}}(\mathcal{F}_{i})\subset \mathcal{F}_{i}\otimes \mathcal{C}_{i-1}\mathcal{H}$, which is killed by $id\otimes g$. 
\end{proof}
\begin{cor}\label{comparison}
Notation as the same as Proposition \ref{hopf_nil}. We have $\omega(\mathcal{E}_{n}^{\vee})= \mathcal{C}_{n-1}\mathcal{O}(\pi_{1}^{nilp}(\mathcal{C}),\omega)$. 
\end{cor}
\begin{proof}
By Lemma \ref{conil}, we have $\Delta(\omega(\mathcal{E}_{n}^{\vee}))\subset \omega(\mathcal{E}_{n}^{\vee}) \otimes \mathcal{C}_{n-1}\mathcal{O}(\pi_{1}^{nilp}(\mathcal{C}),\omega)$. Applying $\epsilon\otimes id$ to this inclusion, and noting $(\epsilon\otimes id)\circ\Delta=id$, we see an inclusion \begin{equation}\label{can_nil}\omega(\mathcal{E}_{n}^{\vee})\subset\mathcal{C}_{n-1}\mathcal{O}(\pi_{1}^{nilp}(\mathcal{C},\omega)).\end{equation} \par
Note the inclusion \eqref{can_nil} is an equality in the case $\mathcal{C}=\mathrm{Rep}^{f}\mathcal{G}(V)$. Indeed, from Example \ref{cohomology_free} and Example \ref{conil_free}, the dimension of both sides of \eqref{can_nil} is equal to $\sum_{i=0}^{n-1}d^{i}$, with $d=\mathrm{dim}\,V$. This forces \eqref{can_nil} to be an equality. We conclude the general case by picking up a faithfully flat morphism $\mathcal{G}(V)\rightarrow \pi_{1}^{nilp}(\mathcal{C},\omega)$ for suitable $V$ (which exists, see below), and using Remark \ref{intersection_univ} and Remark \ref{intersection_conilpotency}. 
\end{proof}
The pro-unipotent affine group scheme $\pi_{1}^{nilp}(\mathcal{C},\omega)$ is an inverse limit $\varprojlim G_{i}$ of unipotent algebraic groups. The corresponding pro-nilpotent Lie algebra is denoted by $\mathfrak{g}$, and is an inverse limit $\varprojlim \mathfrak{g}_{i}$ of nilpotent finite dimensional Lie algebras. We have natural identifications 
\[\begin{split}
\mathrm{Ext}^{1}(1_{\mathcal{C}},1_{\mathcal{C}}) &\cong\mathrm{Hom}_{\mathrm{AffGrp}_{k}}(\pi_{1}^{nilp}(\mathcal{C},\omega),\mathbb{G}_{a}) \cong \varinjlim\mathrm{Hom}_{\mathrm{AffGrp}_{k}}(G_{i},\mathbb{G}_{a})\\&\cong\varinjlim\mathrm{Hom}_{\mathrm{LieAlg}_{k}}(\mathfrak{g}_{i},k)\cong \varinjlim\mathrm{Hom}_{\mathrm{Vect}_{k}}(\mathfrak{g}^{ab}_{i},k).
\end{split}
\]
From these identifications, we see that $\varprojlim_{i}\mathfrak{g}_{i}^{ab}$ is finite dimensional, and can be identified with $\mathrm{Ext}^{1}(1_{\mathcal{C}},1_{\mathcal{C}})^{\vee}$. An arbitrary lifting of vector spaces
\begin{equation}\label{lift_lie}\begin{tikzcd}
&\varprojlim\mathfrak{g}_{i}^{ab} \ar[d,"\exists"] \ar[drr,equal]\\
 &\mathfrak{g} \ar[r,equal]&\varprojlim\mathfrak{g}_{i} \ar[r,twoheadrightarrow] &\varprojlim\mathfrak{g}_{i}^{ab}.
\end{tikzcd}
\end{equation}
gives rise to a surjective morphism $\mathcal{L}(\varprojlim\mathfrak{g}_{i}^{ab})\twoheadrightarrow \mathfrak{g}$ between Lie algebras, 
where the left hand side is the pro-nilpotent completion of the free Lie algebra associated to $\varprojlim\mathfrak{g}_{i}^{ab}$. This corresponds to a faithfully flat morphism $\mathcal{G}(\mathrm{Ext}^{1}(1_{\mathcal{C}},1_{\mathcal{C}}))\rightarrow \pi_{1}^{nilp}(\mathcal{C},\omega)$, and therefore a morphism  \begin{equation}\label{noncan}\zeta':\mathcal{O}(\pi_{1}^{nilp}(\mathcal{C},\omega))\hookrightarrow\mathcal{H}(\mathrm{Ext}^{1}(1_{\mathcal{C}},1_{\mathcal{C}}))\end{equation} between Hopf algebras. Using the same notation as Proposition \ref{hopf_nil}, we have
\[\zeta_{n}': \mathrm{Ext}^{1}(\mathcal{E}_{n},1_{\mathcal{C}})\hookrightarrow \mathrm{Ext}^{1}(1_{\mathcal{C}},1_{\mathcal{C}})^{\otimes n}, n\geq 1\]
by Corollary \ref{comparison}, with $\zeta_{1}'=id$ by construction. A priori $\zeta_{n}'$ may not be canonical, as $\zeta'$ is not (depending on the lifting \eqref{lift_lie}). \par 
Applying $\mathrm{Hom}(\cdot,1_{\mathcal{C}})$ to the defining sequence for $\mathcal{E}_{n}$, from the long exact sequence we get an inclusion 
\begin{equation}\label{induc}\tau_{n}:\mathrm{Ext}^{1}(\mathcal{E}_{n},1_{\mathcal{C}})\hookrightarrow \mathrm{Ext}^{1}(\mathcal{E}_{n-1},1_{\mathcal{C}})\otimes\mathrm{Ext}^{1}(1_{\mathcal{C}},1_{\mathcal{C}}),\end{equation}for $n\geq 2$.
We define $\zeta_{n}: \mathrm{Ext}^{1}(\mathcal{E}_{n},1_{\mathcal{C}})\rightarrow \mathrm{Ext}^{1}(1_{\mathcal{C}},1_{\mathcal{C}})^{\otimes n}$ by settting $\zeta_{1}=id$, and 
\begin{equation}\label{induc'}\zeta_{n}=\tau_{n}\circ (\tau_{n-1}\otimes id)\circ (\tau_{n-2}\otimes id^{\otimes2})\circ...\circ (\tau_{2}\otimes id^{\otimes(n-2)})\end{equation} for $n\geq 2$. 
\begin{lem}We have $\zeta_{n}=\zeta_{n}'$.
\label{compare_grad}
\end{lem}
\begin{proof}
Several natural identifications will be made in the proof:
\begin{enumerate}
\item
Representations of $\pi_{1}^{nilp}(\mathcal{C},\omega)$ are naturally representations of $\mathcal{G}(\mathrm{Ext}^{1}(1_{\mathcal{C}},1_{\mathcal{C}}))$. 
\item For any object $\mathcal{F}\in\mathcal{N}(\mathcal{C})$, we have\[\mathrm{H}^{1}(\pi_{1}^{nilp}(\mathcal{C},\omega),\omega(\mathcal{F}^{\vee}))\cong \mathrm{Ext}^{1}(\mathcal{F},1_{\mathcal{C}}).\]
\item For the category $\mathrm{Rep}^{f}\mathcal{G}(\mathrm{Ext}^{1}(1_{\mathcal{C}},1_{\mathcal{C}}))$, $(\bigoplus_{i=0}^{n-1}\mathrm{Ext}^{1}(1_{\mathcal{C}},1_{\mathcal{C}})^{\otimes i})^{\vee}$, $n\geq 1$ give an iterated universal extension. We have a natural identification
\[\mathrm{H}^{1}(\mathcal{G}(\mathrm{Ext}^{1}(1_{\mathcal{C}},1_{\mathcal{C}})),\bigoplus_{i=0}^{n-1}\mathrm{Ext}^{1}(1_{\mathcal{C}},1_{\mathcal{C}})^{\otimes i})\cong \mathrm{Ext}^{1}(1_{\mathcal{C}},1_{\mathcal{C}})^{\otimes n}.\]
\end{enumerate} \par
For each $n\geq 1$, we have
\[\zeta'_{\leq n}: \omega(\mathcal{E}_{n}^{\vee})\rightarrow \bigoplus_{i=0}^{n-1}\mathrm{Ext}^{1}(1_{\mathcal{C}},1_{\mathcal{C}})^{\otimes i}\] induced by $\zeta'$ after passage to the conilpotency filtration, which is a morphism between representations of $\mathcal{G}(\mathrm{Ext}^{1}(1_{\mathcal{C}},1_{\mathcal{C}}))$. We have a commutative diagram of $\mathcal{G}(\mathrm{Ext}^{1}(1_{\mathcal{C}},1_{\mathcal{C}}))$-representations with exact horizontal sequences:
\begin{equation}
\label{ad}
\begin{tikzcd}[column sep=small, row sep=large]
&0 \ar[r] & \omega(\mathcal{E}_{n}^{\vee}) \ar[r]\ar[d,"\zeta'_{\leq n}"]  &\omega(\mathcal{E}_{n+1}^{\vee}) \ar[r]\ar[d,"\zeta'_{\leq n+1}"] &\mathrm{Ext}^{1}(\mathcal{E}_{n},1_{\mathcal{C}}) \ar[r]\ar[d,"\zeta_{n}'"] &0 \\
&0 \ar[r] &\bigoplus_{i=0}^{n-1}\mathrm{Ext}^{1}(1_{\mathcal{C}},1_{\mathcal{C}})^{\otimes i} \ar[r] &\bigoplus_{i=0}^{n}\mathrm{Ext}^{1}(1_{\mathcal{C}},1_{\mathcal{C}})^{\otimes i} \ar[r] &\mathrm{Ext}^{1}(1_{\mathcal{C}},1_{\mathcal{C}})^{\otimes n} \ar[r] &0.
\end{tikzcd}
\end{equation}
Here $\zeta_{n}'$ should be considered as a morphism between trivial representations. For each $n\geq 1$ , considering connecting morphisms from cohomology in degree $0$ to cohomology in degree $1$ and inflation from group cohomology of $\pi_{1}^{nilp}(\mathcal{C},\omega)$ to group cohomology of $\mathcal{G}(\mathrm{Ext}^{1}(1_{\mathcal{C}},1_{\mathcal{C}}))$, we obtain a commutative diagram from \eqref{ad}
\begin{equation}\label{n+1}
\begin{tikzcd}[row sep=large]
&\mathrm{Ext}^{1}(\mathcal{E}_{n},1_{\mathcal{C}}) \ar[r,equal]\ar[d,equal] &\mathrm{Ext}^{1}(\mathcal{E}_{n},1_{\mathcal{C}})\ar[d,"inflation"]\\
&\mathrm{Ext}^{1}(\mathcal{E}_{n},1_{\mathcal{C}}) \ar[r]\ar[d,"\zeta_{n}'"] &\mathrm{H}^{1}(\mathcal{G}(\mathrm{Ext}^{1}(1_{\mathcal{C}},1_{\mathcal{C}})),\omega(\mathcal{E}_{n}^{\vee}))\ar[d,"\mathrm{H}^{1}(\zeta'_{\leq n})"] \\
&\mathrm{Ext}^{1}(1_{\mathcal{C}},1_{\mathcal{C}})^{\otimes n} \ar[r,equal] &\mathrm{H}^{1}(\mathcal{G}(\mathrm{Ext}^{1}(1_{\mathcal{C}},1_{\mathcal{C}})),\bigoplus_{i=0}^{n-1}\mathrm{Ext}^{1}(1_{\mathcal{C}},1_{\mathcal{C}})^{\otimes i}).\\
\end{tikzcd}
\end{equation}
Considering induced morphisms between cohomology in degree $1$ and inflation, we obtain a commutative diagram from \eqref{ad}
\begin{equation}\label{n}
\begin{tikzcd}[row sep=large,column sep=small]
&\mathrm{Ext}^{1}(\mathcal{E}_{n+1},1_{\mathcal{C}})\ar[r,"\tau_{n+1}"]\ar[d,"inflation"]&\mathrm{Ext}^{1}(\mathcal{E}_{n},1_{\mathcal{C}})\otimes\mathrm{Ext}^{1}(1_{\mathcal{C}},1_{\mathcal{C}})\ar[d,equal]\\
&\mathrm{H}^{1}(\mathcal{G}(\mathrm{Ext}^{1}(1_{\mathcal{C}},1_{\mathcal{C}})),\omega(\mathcal{E}_{n+1}^{\vee}))\ar[d,"\mathrm{H}^{1}(\zeta'_{\leq n+1})"]\ar[r]&\mathrm{Ext}^{1}(\mathcal{E}_{n},1_{\mathcal{C}})\otimes\mathrm{Ext}^{1}(1_{\mathcal{C}},1_{\mathcal{C}})\ar[d,"\zeta_{n}'\otimes id"] \\
&\mathrm{H}^{1}(\mathcal{G}(\mathrm{Ext}^{1}(1_{\mathcal{C}},1_{\mathcal{C}})),\bigoplus_{i=0}^{n}\mathrm{Ext}^{1}(1_{\mathcal{C}},1_{\mathcal{C}})^{\otimes i})\ar[r]&\mathrm{Ext}^{1}(1_{\mathcal{C}},1_{\mathcal{C}})^{\otimes n+1}.\\
\end{tikzcd}
\end{equation}
Replacing $n$ by $n+1$ in \eqref{n+1} and combining it with \eqref{n}, for each $n\geq 1$ we obtain a commutative diagram
\[\begin{tikzcd}[row sep=large]
&\mathrm{Ext}^{1}(\mathcal{E}_{n+1},1_{\mathcal{C}})\ar[r,"\tau_{n+1}"]\ar[d,"\zeta_{n+1}'"]&\mathrm{Ext}^{1}(\mathcal{E}_{n},1_{\mathcal{C}})\otimes\mathrm{Ext}^{1}(1_{\mathcal{C}},1_{\mathcal{C}})\ar[d,"\zeta_{n}'\otimes id"] \\
&\mathrm{Ext}^{1}(1_{\mathcal{C}},1_{\mathcal{C}})^{\otimes n+1} \ar[r,equal] &\mathrm{Ext}^{1}(1_{\mathcal{C}},1_{\mathcal{C}})^{\otimes n+1},
\end{tikzcd}
\]
and by iterating we have 
\[\begin{tikzcd}[column sep=small, row sep=large]
&\mathrm{Ext}^{1}(\mathcal{E}_{n+1},1_{\mathcal{C}})\ar[r,"\tau_{n+1}"]\ar[d,"\zeta_{n+1}'"]&\mathrm{Ext}^{1}(\mathcal{E}_{n},1_{\mathcal{C}})\otimes\mathrm{Ext}^{1}(1_{\mathcal{C}},1_{\mathcal{C}})\ar[d,"\zeta_{n}'\otimes id"]\ar[r,"\tau_{n}\otimes id"] &\mathrm{Ext}^{1}(\mathcal{E}_{n-1},1_{\mathcal{C}})\otimes\mathrm{Ext}^{1}(1_{\mathcal{C}},1_{\mathcal{C}})^{\otimes 2} \ar[d,"\zeta_{n-1}'\otimes id^{\otimes 2}"]\ar[r] &... \\
&\mathrm{Ext}^{1}(1_{\mathcal{C}},1_{\mathcal{C}})^{\otimes n+1} \ar[r,equal]&\mathrm{Ext}^{1}(1_{\mathcal{C}},1_{\mathcal{C}})^{\otimes n+1} \ar[r,equal] &\mathrm{Ext}^{1}(1_{\mathcal{C}},1_{\mathcal{C}})^{\otimes n+1}\ar[r,equal] &... 
\end{tikzcd}
\]
from which we conclude. 
\end{proof}
\begin{lem}
\label{sub_hopf}
Let $J$ be a subspace of $V^{\otimes 2}$. We assume that all symmetric tensors $v\otimes v$, $v\in V$, are contained in $J$. Define $T_{0}=k$ and $T_{1}=V$. For $n\geq 2$, we define $T_{n}$ inductively to be a subspace of $T_{n-1}\otimes V$, as the kernel of the morphism
\[\begin{tikzcd}[column sep=small] &T_{n-1}\otimes V\ar[r,hook] &T_{n-2}\otimes V \otimes V \ar[r]&T_{n-2}\otimes (V^{\otimes 2}/J)\end{tikzcd}\]
by taking quotient $V^{\otimes 2}\twoheadrightarrow V^{\otimes 2}/J$ for the last two tensor factors. Then  $\bigoplus_{n\geq 0}T_{n}$ has a Hopf algebra structure inherited from $\mathcal{H}(V)$, denoted by $\mathcal{H}(V,J)$.  
\end{lem}
\begin{proof}
This can be checked directly.
\end{proof}
To get a better approximation of $\pi_{1}^{nilp}(\mathcal{C},\omega)$, we place ourselves in the following situation.
\begin{noco}
Embed $\mathcal{C}$ into an another $k$-linear abelian tensor category $\mathcal{D}$ fully faithfully, such that
\begin{enumerate}
\item $\mathcal{C}$ is closed under taking subquotients and extensions in $\mathcal{D}$.
\item $\mathcal{D}$ has enough injectives. 
\end{enumerate}
The $\mathrm{Ext}$ groups $\mathrm{Ext}^{i}$, $i\geq 0$ can be defined either as the group of Yoneda extensions in $\mathcal{D}$, or defined as derived functors.
\end{noco}
\begin{lem}\label{cup}
Apply $\mathrm{Hom}(\cdot,1_{\mathcal{C}})$ to the short exact sequence \eqref{univ_def} defining the universal extension $U(\mathcal{E})$ of an object $\mathcal{E}\in \mathcal{C}$. For each $i$, the resulting connecting homomorphism 
\[\mathrm{Ext}^{i}(\mathcal{E},1_{\mathcal{C}})\otimes\mathrm{Ext}^{1}(1_{\mathcal{C}},1_{\mathcal{C}}) \rightarrow \mathrm{Ext}^{i+1}(\mathcal{E},1_{\mathcal{C}})\]
in the associated long exact sequence is given by the Yoneda cup product, and differs from the cup product $\cup$ in the sense of derived functors by a sign $(-1)^{i}$. 
\end{lem}
\begin{proof}
The first part is by construction. The second part follows from Chapitre \uppercase\expandafter{\romannumeral3} 3.2.5 \cite{AST_1996__239__R1_0}.
\end{proof}
\begin{noco}
In the general construction of Lemma \ref{sub_hopf}, we take $V=\mathrm{Ext}^{1}(1_{\mathcal{C}},1_{\mathcal{C}})$ and $J$ to be the kernel of the cup product 
\[\cup:\mathrm{Ext}^{1}(1_{\mathcal{C}},1_{\mathcal{C}})\otimes \mathrm{Ext}^{1}(1_{\mathcal{C}},1_{\mathcal{C}}) \rightarrow \mathrm{Ext}^{2}(1_{\mathcal{C}},1_{\mathcal{C}}).\] The resulting Hopf algebra $\mathcal{H}(V,J)$ will be denoted by $\mathcal{H}(\mathrm{Ext}^{\leq 2}(1_{\mathcal{C}},1_{\mathcal{C}}))$. 
\end{noco}
\begin{prop}
\label{mor_app}
There are natural inclusions $\xi_{n}$
\[
\begin{tikzcd}
&\mathrm{Ext}^{1}(\mathcal{E}_{n},1_{\mathcal{C}}) \ar[rr, bend left=20, "\zeta_{n}",hook]\ar[r,hook,"\xi_{n}"] &T_{n} \ar[r,hook] &\mathrm{Ext}^{1}(1_{\mathcal{C}},1_{\mathcal{C}})^{\otimes n}
\end{tikzcd}
\]
for $n\geq 1$. (The notation $\zeta_{n}$ is from \eqref{induc'} and $T_{n}$ is from Proposition \ref{sub_hopf}.)
\end{prop}
\begin{proof}
In the sequel, the morphism $\tau_{n}$ is the one from \eqref{induc}. \par
Apply induction on $n$. It is clear for $n=1$. \par
For $n=2$, applying $\mathrm{Hom}(\cdot,1_{\mathcal{C}})$ to the exact sequence defining $\mathcal{E}_{2}$, we have an exact sequence
\[
\begin{tikzcd}
&0\ar[r] &\mathrm{Ext}^{1}(\mathcal{E}_{2},1_{\mathcal{C}}) \ar[r,"\tau_{2}"] &\mathrm{Ext}^{1}(\mathcal{E}_{1},1_{\mathcal{C}})\otimes \mathrm{Ext}^{1}(1_{\mathcal{C}},1_{\mathcal{C}}) \ar[r,"-\cup"] &\mathrm{Ext}^{2}(\mathcal{E}_{1},1_{\mathcal{C}}).
\end{tikzcd}
\]
Note that $\mathcal{E}_{1}$ is $1_{\mathcal{C}}$, so $\mathrm{Ext}^{1}(\mathcal{E}_{2},1_{\mathcal{C}})$ is exactly defined as the kernel of the cup product, which is $T_{2}$. \par 
For $n\geq 2$, we have the following commutative diagram whose horizontal and vertical sequences are all exact:
\label{diagram}
\begin{equation}
\label{diagram}
\begin{tikzcd}[column sep=small,row sep=large]
&{}&{} &{} &\mathrm{Ext}^{2}(\mathcal{E}_{n-1},1_{\mathcal{C}}) \ar[d] \\
&0\ar[r]&\mathrm{Ext}^{1}(\mathcal{E}_{n+1},1_{\mathcal{C}}) \ar[r,"\tau_{n+1}"] &\mathrm{Ext}^{1}(\mathcal{E}_{n},1_{\mathcal{C}})\otimes \mathrm{Ext}^{1}(1_{\mathcal{C}},1_{\mathcal{C}}) \ar[d,hook,"\tau_{n}\otimes id"]\ar[r,"-\cup"] & \mathrm{Ext}^{2}(\mathcal{E}_{n},1_{\mathcal{C}}) \ar[d]  \\
&{}&{}&\mathrm{Ext}^{1}(\mathcal{E}_{n-1},1_{\mathcal{C}})\otimes\mathrm{Ext}^{1}(1_{\mathcal{C}},1_{\mathcal{C}})\otimes\mathrm{Ext}^{1}(1_{\mathcal{C}},1_{\mathcal{C}}) \ar[r,"-id\otimes \cup"]& \mathrm{Ext}^{1}(\mathcal{E}_{n-1},1_{\mathcal{C}})\otimes \mathrm{Ext}^{2}(1_{\mathcal{C}},1_{\mathcal{C}}) .\\
\end{tikzcd}
\end{equation}
Here the horizontal (resp. vertical) exact sequence is a part of the long exact sequence obtained by applying $\mathrm{Hom}(\cdot,1_{\mathcal{C}})$ to the short exact sequence defining $\mathcal{E}_{n+1}$ (resp. $\mathcal{E}_{n}$). Now $\mathrm{Ext}^{1}(\mathcal{E}_{n+1},1_{\mathcal{C}})$ is canonically contained in the kernel of the composition going from $\mathrm{Ext}^{1}(\mathcal{E}_{n},1_{\mathcal{C}})\otimes \mathrm{Ext}^{1}(1_{\mathcal{C}},1_{\mathcal{C}})$ to $\mathrm{Ext}^{1}(\mathcal{E}_{n-1},1_{\mathcal{C}})\otimes \mathrm{Ext}^{2}(1_{\mathcal{C}},1_{\mathcal{C}})$, which can be identified as a subspace of $T_{n+1}$ by induction hypothesis. 
\end{proof}
\begin{cor}
\label{non_qua}
There is a non-canonical inclusion \begin{equation}\label{refine}\xi': \mathcal{O}(\pi_{1}^{nilp}(\mathcal{C}),\omega)\hookrightarrow \mathcal{H}(\mathrm{Ext}^{\leq 2}(1_{\mathcal{C}},1_{\mathcal{C}}))\end{equation} between Hopf algebras such that, $\xi_{n}$ in Proposition \ref{mor_app} are isomorphisms for all $n\geq 1$ if and only if $\xi'$ from \eqref{refine} is an isomorphism.
\end{cor}
\begin{proof}
There is an inclusion \[\zeta':\mathcal{O}(\pi_{1}^{nilp}(\mathcal{C}),\omega)\hookrightarrow \mathcal{H}(\mathrm{Ext}^{1}(1_{\mathcal{C}},1_{\mathcal{C}}))\] as that of \eqref{noncan}. By Lemma \ref{compare_grad} and Proposition \ref{mor_app}, we see a factorization 
\[
\begin{tikzcd}
&\mathcal{O}(\pi_{1}^{nilp}(\mathcal{C},\omega)) \ar[rr, bend left=20, "\zeta'",hook]\ar[r,hook,"\xi'"] &\mathcal{H}(\mathrm{Ext}^{\leq 2}(1_{\mathcal{C}},1_{\mathcal{C}})) \ar[r,hook] &\mathcal{H}(\mathrm{Ext}^{1}(1_{\mathcal{C}},1_{\mathcal{C}})).
\end{tikzcd}
\]The inclusion $\xi'$ is not canonical, but the induced morphisms between associated graded pieces with respect to the conilpotency filtration are given by those canonical $\xi_{n}$. 
\end{proof}
The following corollary will be used in the last section. 
\begin{cor}
\label{null_inter}
The inclusions $\xi_{n}$ for all $n$ are isomorphisms, if and only if, for each $n\geq 1$, the images of the cup product $\cup: \mathrm{Ext}^{1}(\mathcal{E}_{n},1_{\mathcal{C}})\otimes \mathrm{Ext}^{1}(1_{\mathcal{C}},1_{\mathcal{C}})\rightarrow \mathrm{Ext}^{2}(\mathcal{E}_{n},1_{\mathcal{C}})$ and the morphism $\mathrm{Ext}^{2}(\mathcal{E}_{n-1},1_{\mathcal{C}})\rightarrow \mathrm{Ext}^{2}(\mathcal{E}_{n},1_{\mathcal{C}})$ have intersection $0$.
\end{cor}
\begin{proof}
This follows from the argument of Proposition \ref{mor_app}, especially the diagram \eqref{diagram}. 
\end{proof}
\begin{emp}
Let $X$ be a geometrically connected smooth curve over $k$. Consider the category $\mathcal{C}$ of étale $\mathbb{Q}_{\ell}$-local systems on $X_{\bar{k}}$ (resp. the category of integrable connections on $X$). The fiber functor $\omega$ is given by a point $x\in X(k)$.\par
In \cite{10.1215/00127094-1444296} the author further assumes $X$ to be affine. The reason to restrict to affine curves is that all cohomology groups in degree $2$ vanish, and diagram \eqref{diagram} is greatly simplified. It turns out that the Tannkian group scheme of nilpotent $\mathbb{Q}_{\ell}$-local systems (resp. integrable connections) is the pro-unipotent completion of $\mathcal{G}_{d}$ over $\mathbb{Q}_{\ell}$ (resp. over $k$), with $d=\mathrm{dim}\,\mathrm{H}^{1}_{et}(X_{\bar{k}},\mathbb{Q}_{\ell})$ (resp. $d=\mathrm{dim}\,\mathrm{H}^{1}_{dR}(X)$). \par
The case of smooth proper curves is treated as Proposition 3.4 \cite{10.1215/00127094-3146817}. The argument is more tricky due to possibly non-vanishing $\mathrm{H}^{2}$ and non-vanishing cup products. Still it is possible to prove algebraically that all $\xi_{n}$ as in Proposition \ref{mor_app} are isomorphisms, and the Hopf algebra of the Tannkian group scheme of nilpotent $\mathbb{Q}_{\ell}$-local systems (resp. integrable connections) is $\mathcal{H}(\mathrm{H}^{\leq 2}_{et}(X,\mathbb{Q}_{l}))$ (resp. $\mathcal{H}(\mathrm{H}_{dR}^{\leq 2}(X)$)). \par
For a geometrically connected, smooth and projective variety $X$ with a base point $x\in X(k)$, it is still possible to prove that the Hopf algebra of the Tannkian group scheme of nilpotent $\mathbb{Q}_{\ell}$-local systems (resp. integrable connections) is $\mathcal{H}(\mathrm{H}^{\leq 2}_{et}(X,\mathbb{Q}_{l}))$ (resp. $\mathcal{H}(\mathrm{H}_{dR}^{\leq 2}(X)$)). Indeed, by comparisons among different cohomology theories, it suffices to treat $\pi_{1}^{dR}$ for a smooth projective complex variety. This is classical but the proof is Hodge theoretic (see Theorem \ref{case_deRham} below). It is not clear to the author if there is a purely algebraic proof. 
\end{emp}
The following example suggests that the inclusions $\xi_{n}$ in Proposition \ref{mor_app} may not be isomorphisms.
\begin{emp} 
Consider the Heisenberg group $\mathcal{G}$ of $3\times 3$ upper triangular unipotent matrices
\[\begin{pmatrix}
1&*&* \\
0&1&* \\
0&0&1 
\end{pmatrix}.\] Let $\mathcal{C}$ be the category of finite dimensional $k$-linear representations of $G$, $\mathcal{D}$ be the category of all $k$-linear representations, and $\omega$ be the forgetful functor to $\mathrm{Vect}_{k}$. Let $\mathcal{E}_{1}$ be the $1$-dimensional trivial representation. It is easy to see that $V=\mathrm{H}^{1}(G,\mathcal{E}_{1})$ is $2$-dimensional. Assume all $\xi_{n}$ as in Proposition \ref{mor_app} are isomorphisms, and then the Hopf algebra of $\mathcal{G}$ is $\mathcal{H}(V,J)$ where $J$ is the kernel of the cup product $\cup:V^{\otimes 2}\rightarrow \mathrm{H}^{2}(G,\mathcal{E}_{2})$. There are two possibilities:
\begin{enumerate}
\item $J$ is linearly spanned by symmetric tensors. In this case we see that $\mathcal{H}(V,J)$ is the Hopf algebra of $\mathbb{G}_{a}\times_{k}\mathbb{G}_{a}$, which is impossible as $\mathcal{G}$ is not commutative.
\item $J=V^{\otimes 2}$. In this case we see that $\mathcal{G}$ is isomorphic to $\mathcal{G}(V)$. This is impossible since $\mathcal{G}$ is an algebraic group but $\mathcal{G}(V)$ is not.
\end{enumerate}
\end{emp}

\section{A study of $\pi_{1}^{N}$ via Hodge theory} 
Let $X$ be a geometrically connected, smooth and proper variety over $k$, with a base point $x\in X(k)$. We apply the general construction from last section to the following cases and we fix some notations:
\begin{enumerate}
\item Let $\mathcal{C}$ be the category of coherent $\mathcal{O}_{X}$-modules, $\omega$ be the fiber functor defined by $x$, and $\mathcal{D}$ be the category of quasi-coherent $\mathcal{O}_{X}$-modules. We fix an iterated universal extension
\[\begin{tikzcd}
&{...}\ar[r] &(\mathcal{E}_{n},e_{n}) \ar[r] &{...} \ar[r] &(\mathcal{E}_{1},e_{1}). 
\end{tikzcd}
\]
For each $n\geq 2$, we have a morphism $\tau_{n}: \mathrm{H}^{1}(X,\mathcal{E}_{n}^{\vee})\hookrightarrow \mathrm{H}^{1}(X,\mathcal{E}_{n-1}^{\vee})\otimes\mathrm{H}^{1}(X)$ defined as \eqref{induc}. For $n\geq 1$ we have a subspace $T_{n}\hookrightarrow \mathrm{H}^{1}(X)^{\otimes n}$ defined in Lemma \ref{sub_hopf}, an inclusion $\zeta_{n}:\mathrm{H}^{1}(X,\mathcal{E}_{n}^{\vee})\hookrightarrow \mathrm{H}^{1}(X)^{\otimes n}$ defined as \eqref{induc'}, which factors through $T_{n}$ as $\xi_{n}:\mathrm{H}^{1}(X,\mathcal{E}_{n}^{\vee})\hookrightarrow T_{n}$ by Proposition \ref{mor_app}. We have a Hopf algebra $\mathcal{H}(\mathrm{H}^{\leq 2}(X))$.
\item Let $\mathcal{C}$ be the category of integrable connections on $X$, $\omega$the fiber functor defined by $x$, and $\mathcal{D}$ be the category of $\mathcal{D}_{X}$-modules. We fix an iterated universal extension
\[\begin{tikzcd}
&{...}\ar[r] &(\mathcal{F}_{n},f_{n}) \ar[r] &{...} \ar[r] &(\mathcal{F}_{1},f_{1}). 
\end{tikzcd}
\]
For each $n\geq 2$, we have a morphism $\tau_{dR,n}: \mathrm{H}^{1}_{dR}(X,\mathcal{F}_{n}^{\vee})\hookrightarrow \mathrm{H}^{1}_{dR}(X,\mathcal{F}_{n-1}^{\vee})\otimes\mathrm{H}^{1}_{dR}(X)$ defined as \eqref{induc}. For $n\geq 1$ we have a subspace $T_{dR,n}\hookrightarrow \mathrm{H}^{1}_{dR}(X)^{\otimes n}$ defined in Lemma \ref{sub_hopf}, an inclusion $\zeta_{dR,n}:\mathrm{H}^{1}_{dR}(X,\mathcal{F}_{n}^{\vee})\hookrightarrow \mathrm{H}^{1}_{dR}(X)^{\otimes n}$ defined as \eqref{induc'}, which factors through $T_{dR,n}$ as $\xi_{dR,n}:\mathrm{H}_{dR}^{1}(X,\mathcal{F}_{n}^{\vee})\hookrightarrow T_{dR,n}$ by \ref{mor_app}. We have a Hopf algebra $\mathcal{H}(\mathrm{H}_{dR}^{\leq 2}(X))$.
\end{enumerate} \par
Here is our main result concerning Nori's fundamental group of nilpotent bundles. 
\begin{thm}
\label{quadric}
Let $X$ be a geometrically connected, smooth and proper variety over $k$, with a base point $x\in X(k)$. The Hopf algebra $\mathcal{O}(\pi_{1}^{N}(X,x))$ is isomorphic to $\mathcal{H}(\mathrm{H}^{\leq 2}(X))$. 
\end{thm} 
We immediately reduce to the case of smooth projective complex varieties.
\begin{proof}[A reduction for Theorem \ref{quadric}]
The isomorphism in Theorem \ref{quadric} is given by the non-canonical $\xi'$ as in Corollary \ref{non_qua}, and we are forced to show that the resulting canonical $\xi_{n}$ are all isomorphisms. Our formulation is clearly compatible with extension of scalars, and therefore we can assume $k=\mathbb{C}$. We can find a proper birational morphism $\pi: \tilde{X}\rightarrow X$ with $\tilde{X}$ being smooth projective. Using $\mathrm{R}\pi_{*}\mathcal{O}_{\tilde{X}}\cong\mathcal{O}_{X}$, we can show that $\pi$ induces an isomorphism between Nori's fundamental groups. Replacing $X$ by $\tilde{X}$, we reduce to the case that $X$ is a smooth projective complex variety.
\end{proof}
Let $X$ be a smooth  projective complex variety, with $x\in X(\mathbb{C})$. We fix an ample line bundle $\mathcal{L}$ on $X$. The line bundle $\mathcal{L}$ defines a Kahler class on $X^{an}$.\par
\begin{thm}\label{case_deRham}
There exists a non-canonical isomorphism 
\[ \mathcal{O}(\pi_{1}^{dR}(X,x))\cong\mathcal{H}(\mathrm{H}_{dR}^{\leq 2}(X)).\]
In particular, for each $n\geq 1$, $\xi_{dR,n}$ is an isomorphim.
\end{thm}
\begin{proof}
We can express $\mathcal{O}(\pi_{1}^{dR}(X,x))$ in terms of $\Omega^{*}(X^{an})$ (originally in \cite{chen1977iterated}, or 3.5 \cite{gil2017multiple}). By the formality of compact Kähler manifolds (originally in \cite{deligne1975real}, or Theorem 3.13 \cite{amoros1996fundamental}), we can replace $\Omega^{*}(X^{an})$ by $\mathrm{H}^{*}_{dR}(X)$ , and the resulting Hopf algebra is exactly $\mathcal{H}(\mathrm{H}_{dR}^{\leq 2}(X))$. Now the dimensions of $\mathcal{C}_{n}\mathcal{O}(\pi_{1}^{dR}(X,x))$ and $\mathcal{C}_{n}\mathcal{H}(\mathrm{H}_{dR}^{\leq 2}(X))$ are the same, and this forces $\xi_{n}$ be to isomorphisms for all $n\geq 1$. 
\end{proof}

\begin{rmk}[\cite{deligne1975real}]Let $X$ be a compact Kähler manifold.
\begin{enumerate}
\item  The pro-unipotent completion of $\pi_{1}^{top}(X^{an},x)$ over $\mathbb{R}$ is determined by 
\[\cup:\mathrm{H}^{1}_{dR}(X^{an},\mathbb{R})\otimes\mathrm{H}^{1}_{dR}(X^{an},\mathbb{R})\rightarrow \mathrm{H}^{2}_{dR}(X^{an},\mathbb{R}).\] Therefore it is sufficient to work over $\mathbb{R}$ rather than $\mathbb{C}$. 
\item We do not make an exhaustive use of $\mathrm{H}^{*}_{dR}(X^{an},\mathbb{R})$. For example, cohomology groups in degree $\geq 3$ are totally missing from our discussion. The full cohomology algebra  $\mathrm{H}^{*}_{dR}(X^{an},\mathbb{R})$ can be used to determine real homotopy groups. 
\end{enumerate}
\end{rmk}

\begin{noco}
\begin{enumerate}
\item Forgetting the connection map in $\mathcal{F}_{n}$ we obtain a nilpotent bundle of depth $\leq n$, denoted by $\underline{\mathcal{F}_{n}}$. By universal property of $\mathcal{E}_{n}$ we have a unique morphism $\mathcal{E}_{n}\rightarrow \underline{\mathcal{F}_{n}}$ sending $e_{n}$ to $f_{n}$, and taking dual we have
\[\beta_{n}: \underline{\mathcal{F}_{n}^{\vee}} \rightarrow \mathcal{E}_{n}^{\vee}. \]
\item
We write $\rho_{n}$ for the natural morphism
\[\mathrm{H}^{1}_{dR}(X,\mathcal{F}_{n}^{\vee})\rightarrow \mathrm{H}^{1}(X,\underline{\mathcal{F}^{\vee}_{n}})\]
coming from the Hodge-de Rham spectral sequence. 
\item
For a morphism between objects (which can be sheaves, or connections) $\alpha:\mathcal{E}\rightarrow \mathcal{F}$, we write \[\mathrm{H}^{*}(\alpha):\mathrm{H}^{*}(\mathcal{E})\rightarrow\mathrm{H}^{*}(\mathcal{F})\] for the induced morphisms between cohomology groups.
\end{enumerate}
\end{noco}
\begin{proof}[Proof of Theorem \ref{quadric}] 
Dualizing the defining exact sequence of $\mathcal{F}_{n}$, and forgetting all connection maps, we have a short exact sequence of nilpotent vector bundles
\begin{equation}\label{dn4}
\begin{tikzcd} &0 \ar[r] &\underline{\mathcal{F}_{n-1}^{\vee}}\ar[r] &\underline{\mathcal{F}_{n}^{\vee}} \ar[r] &\mathrm{H}_{dR}^{1}(X,\mathcal{F}_{n-1}^{\vee})\otimes\mathcal{O}_{X}\ar[r] &0,\end{tikzcd}
\end{equation}
and the connecting morphism from $\mathrm{H}^{0}$ to $\mathrm{H}^{1}$ coincides with $\rho_{n-1}$. The sequence \eqref{dn4} can be completed into a commutative diagram
\begin{equation}\label{dn5}
\begin{tikzcd} [row sep=large]
&0 \ar[r] &\underline{\mathcal{F}_{n-1}^{\vee}}\ar[r]\ar[d,"\beta_{n-1}"] &\underline{\mathcal{F}_{n}^{\vee}} \ar[d,"\beta_{n}"]\ar[r] &\mathrm{H}_{dR}^{1}(X,\mathcal{F}_{n-1}^{\vee})\otimes\mathcal{O}_{X}\ar[r]\ar[d,"g\otimes id"] &0 \\
&0 \ar[r] &\mathcal{E}_{n-1}^{\vee}\ar[r] &\mathcal{E}_{n}^{\vee} \ar[r] &\mathrm{H}^{1}(X,\mathcal{E}_{n-1}^{\vee})\otimes\mathcal{O}_{X}\ar[r] &0.
\end{tikzcd}
\end{equation}
To see what $g$ is, in \eqref{dn5} we consider connecting morphisms from $\mathrm{H}^{0}(X,\cdot)$ to $\mathrm{H}^{1}(X,\cdot)$, and there is a commutative diagram
\[\begin{tikzcd}
&\mathrm{H}^{1}_{dR}(X,\mathcal{F}_{n-1}^{\vee}) \ar[d,"g"]\ar[r,"\rho_{n-1}"] &\mathrm{H}^{1}(X,\underline{\mathcal{F}^{\vee}_{n-1}}) \ar[d,"\mathrm{H}^{1}(\beta_{n-1})"]\\
&\mathrm{H}^{1}(X,\mathcal{E}_{n-1}^{\vee}) \ar[r,equal]&\mathrm{H}^{1}(X,\mathcal{E}_{n-1}^{\vee}).
\end{tikzcd}
\]
Therefore we have $g=\mathrm{H}^{1}(\beta_{n-1})\circ\rho_{n-1}$. Now applying $\mathrm{H}^{1}(X,\cdot)$ to \eqref{dn5} we get a commutative square
\begin{equation}
\label{jiu}
\begin{tikzcd}[row sep=large]
&\mathrm{H}^{1}(X,\underline{\mathcal{F}_{n}^{\vee}}) \ar[r]\ar[d,"\mathrm{H}^{1}(\beta_{n})"] &\mathrm{H}_{dR}^{1}(X,\mathcal{F}_{n-1}^{\vee})\otimes\mathrm{H}^{1}(X)\ar[d,"\mathrm{H}^{1}(\beta_{n-1})\circ\rho_{n-1} \otimes id"] \\
&\mathrm{H}^{1}(X,\mathcal{E}_{n}^{\vee}) \ar[r,"\tau_{n}"] &\mathrm{H}^{1}(X,\mathcal{E}_{n-1}^{\vee})\otimes\mathrm{H}^{1}(X).
\end{tikzcd}
\end{equation}
We have the obvious commutative square
\begin{equation}
\label{shi}
\begin{tikzcd}[row sep=large]
&\mathrm{H}^{1}_{dR}(X,\mathcal{F}_{n}^{\vee})  \ar[r,"\tau_{dR,n}"]\ar[d,"\rho_{n}"] &\mathrm{H}_{dR}^{1}(X,\mathcal{F}_{n-1}^{\vee})\otimes\mathrm{H}_{dR}^{1}(X)\ar[d,"id\otimes\rho_{1}"] \\
&\mathrm{H}^{1}(X,\underline{\mathcal{F}_{n}^{\vee}}) \ar[r] &\mathrm{H}_{dR}^{1}(X,\mathcal{F}_{n-1}^{\vee})\otimes\mathrm{H}^{1}(X).
\end{tikzcd}
\end{equation}
Combining \eqref{shi}\eqref{jiu}, we obtain a commutative square
\[
\begin{tikzcd}
&\mathrm{H}^{1}_{dR}(X,\mathcal{F}_{n}^{\vee})  \ar[r,"\tau_{dR,n}"]\ar[d,"\mathrm{H}^{1}(\beta_{n})\circ\rho_{n}"] &\mathrm{H}_{dR}^{1}(X,\mathcal{F}_{n-1}^{\vee})\otimes\mathrm{H}_{dR}^{1}(X)\ar[d,"\mathrm{H}^{1}(\beta_{n-1})\circ\rho_{n-1}\otimes\rho_{1}"] \\
&\mathrm{H}^{1}(X,\mathcal{E}_{n}^{\vee}) \ar[r,"\tau_{n}"] &\mathrm{H}^{1}(X,\mathcal{E}_{n-1}^{\vee})\otimes\mathrm{H}^{1}(X),
\end{tikzcd}
\]
and therefore a commutative diagram
\begin{equation}\label{fin}
\begin{tikzcd}[row sep=large]
&\mathrm{H}^{1}_{dR}(X,\mathcal{F}_{n}^{\vee})  \ar[r,"\xi_{dR,n}",hook]\ar[d,"\mathrm{H}^{1}(\beta_{n})\circ\rho_{n}"] &T_{dR,n} \ar[r,hook]\ar[d,"\rho_{1}^{\otimes n}|_{T_{dR,n}}"] &\mathrm{H}^{1}_{dR}(X)^{\otimes n}\ar[d,"\rho_{1}^{\otimes n}"] \\
&\mathrm{H}^{1}(X,\mathcal{E}_{n}^{\vee})  \ar[r,"\xi_{n}",hook] &T_{n} \ar[r,hook] &\mathrm{H}^{1}(X)^{\otimes n}
\end{tikzcd}
\end{equation} by iterating.  Using the Kähler class associated to $\mathcal{L}$,  we find $\iota_{i}$
\[\begin{tikzcd}
&\mathrm{H}^{i}(X)\ar[r,"\iota_{i}"]\ar[rr,bend left=20,"id"] &\mathrm{H}^{i}_{dR}(X) \ar[r] &\mathrm{H}^{i}(X),
\end{tikzcd}
\] for $i=1,2$, lifting the canonical projection $\mathrm{H}^{i}_{dR}(X) \rightarrow\mathrm{H}^{i}(X)$ defined by the Hodge-de Rham spectral sequence. These liftings are compatible with cup products in the sense of the following commutative square
\[\begin{tikzcd}[row sep=large]
&\mathrm{H}^{1}(X)\otimes\mathrm{H}^{1}(X) \ar[r,"\cup"]\ar[d,"\iota_{1}\otimes\iota_{1}"] &\mathrm{H}^{2}(X) \ar[d,"\iota_{2}"] \\
&\mathrm{H}^{1}_{dR}(X)\otimes\mathrm{H}_{dR}^{1}(X) \ar[r,"\cup"]&\mathrm{H}^{2}_{dR}(X).
\end{tikzcd}
\]
This gives a lifting 
\[\begin{tikzcd}[column sep=large, row sep=large]
&T_{n} \ar[dr,equal] \ar[d,"\iota"] \\
&T_{dR,n} \ar[r,"\rho_{1}^{\otimes n}|_{T_{dR,n}}"] &T_{n},
\end{tikzcd}
\]
so the morphism 
 \[\rho_{1}^{\otimes n}|_{T_{dR,n}}: T_{dR,n}\rightarrow T_{n}\] 
is surjective. By Theorem \ref{case_deRham} we have $\xi_{dR,n}$ to be an isomorphism. From \eqref{fin} we see that $\xi_{n}$ is surjective and therefore an isomorphism. By Corollary \ref{non_qua} we are done.
\end{proof}
\section{Group cohomology of a pro-unipotent group scheme}
Let $(X,x)$ be a pointed smooth projective variety over a field of characteristic $0$. In this section we aim to determine low degree group cohomology of $\pi_{1}^{N}(X,x)$ and $\pi_{1}^{dR}(X,x)$:
\begin{thm}
\label{cohomology}
Let $(X,x)$ be a pointed geometrically connected smooth proper variety over a field $k$ of characteristic $0$. The low degree group cohomology of the trivial representation of $\pi_{1}^{N}(X,x)$ (resp. $\pi_{1}^{dR}(X,x)$) is given by 
\[\mathrm{H}^{i}=
\begin{cases}
k, &\text{ for } i=0, \\
\mathrm{H}^{1}(X)(\text{resp. }\mathrm{H}_{dR}^{1}(X)) &\text{ for } i=1, \\
\mathrm{Im} (\cup:\mathrm{H}^{1}(X)^{\otimes 2}\rightarrow \mathrm{H}^{2}(X))(\text{resp. }\mathrm{Im} (\cup:\mathrm{H}_{dR}^{1}(X)^{\otimes 2}\rightarrow \mathrm{H}_{dR}^{2}(X))) &\text{ for } i=2.\\
\end{cases}\]
\end{thm}
Bofore the proof we give some remarks and immediate corollaries. 
\begin{rmk}\label{relation}
It is clear that $\pi_{1}^{N}(X,x)$ (resp. $\pi_{1}^{dR}(X,x)$) can be realized as a quotient of a free pro-unipotent group scheme,
\[\mathcal{G}(V)\rightarrow  \pi_{1}^{N}(X,x),\]
for $V=\mathrm{H}^{1}(X)$ (resp. $V=\mathrm{H}^{1}_{dR}(X)$),  which is a so-called proper presentation (we recall Lemma \ref{free_uni} for a Hopf algebra $\mathcal{H}(V)$ and the pro-unipotent group scheme $\mathcal{G}(V)$ associated to it. We refer to Definition 3.10 \cite{LUBOTSKY198276} for the definition of a proper presentation). We use $\mathcal{G}'$ to denote the kernel of this quotient. The dimension of the second group cohomology is the minimal possible $n$, such that there are elements $g_{1},...,g_{n}$ of $\mathcal{G}'$, such that the abstract subgroup of $\mathcal{G}'$ generated by all $\mathcal{G}(V)$-conjugates of $g_{1},...,g_{n}$ is Zariski dense in $\mathcal{G}'$. We refer to Theorem 3.11 \cite{LUBOTSKY198276}. 
\end{rmk}
\begin{rmk}
In case $X$ is a smooth proper curve, $\pi_{1}^{N}(X,x)$ is a free pro-unipotent group scheme and the computation has been made in Example \ref{cohomology_free}. In case $X$ is an abelian variety and thus $\pi_{1}^{N}(X,x)$ is a commutative unipotent group scheme, the computation is also classical (see the lemma at Page 85 \cite{jantzen2003representations}). 
\end{rmk}
\begin{rmk}
Let $X$ be a smooth proper curve over $k$. Theorem \ref{cohomology} for the de Rham fundamental group $\pi_{1}^{dR}(X,x)$ can be understood as a de Rham $K(\pi,1)$-property for $X$. See the last section of \cite{BAO2025103646}.
\end{rmk}

\begin{proof}[Proof of Theorem \ref{cohomology}]
The cohomology in degree $0$ and $1$ should be clear. In degree $2$, we prove for $\pi_{1}^{N}(X,x)$ and the de Rham case is similar. We use the same notation as that fixed at the beginning of Section 3. We have $\mathcal{E}_{0}\cong \mathcal{O}_{X}$, and $
\mathcal{E}_{0,x}^{\vee}$ underlines the trivial representation of $\pi_{1}^{N}(X,x)$. The notations $\mathcal{E}_{0}$ and $\mathcal{O}_{X}$ are used interchangeably. We have a short exact sequence of $\pi_{1}^{N}(X,x)$-representations:
\[\begin{tikzcd}
& 0\ar[r]&\mathcal{E}^{\vee}_{0,x} \ar[r] &\varinjlim_{n} \mathcal{E}^{\vee}_{n,x} \ar[r] &\varinjlim_{n} \mathcal{E}^{\vee}_{n,x}/\mathcal{E}^{\vee}_{0,x} \ar[r] &0. 
\end{tikzcd}\]
Note $\varinjlim_{n} \mathcal{E}^{\vee}_{n,x}$ is the Hopf algebra of $\pi_{1}^{N}(X,x)$, regarded as the regular representation, and thus the group cohomology vanishes in degree $\geq 1$. Taking the long exact sequence of cohomology we have
\[\begin{split}\mathrm{H}^{2}(\pi_{1}^{N}(X,x),\mathcal{E}_{0,x}^{\vee})&\cong\mathrm{H}^{1}(\pi_{1}^{N}(X,x),\varinjlim_{n} \mathcal{E}^{\vee}_{n,x}/\mathcal{E}^{\vee}_{0,x})\\
&\cong\varinjlim_{n}\mathrm{H}^{1}(\pi_{1}^{N}(X,x),\mathcal{E}^{\vee}_{n,x}/\mathcal{E}^{\vee}_{0,x}) \\
&\cong \varinjlim_{n}\mathrm{H}^{1}(X,\mathcal{E}_{n}^{\vee}/\mathcal{E}_{0}^{\vee}).
\end{split}\]
For each $n\geq 1$ the morphism $\mathrm{H}^{1}(X,\mathcal{E}_{n}^{\vee}/\mathcal{E}_{0}^{\vee})\rightarrow \mathrm{H}^{1}(X,\mathcal{E}_{n+1}^{\vee}/\mathcal{E}_{0}^{\vee})$ is induced by the short exact sequence
\[\begin{tikzcd}
&0 \ar[r] &\mathcal{E}_{n}^{\vee}/\mathcal{E}_{0}^{\vee} \ar[r] &\mathcal{E}_{n+1}^{\vee}/\mathcal{E}_{0}^{\vee} \ar[r] &\mathrm{H}^{1}(X,\mathcal{E}_{n}^{\vee})\otimes \mathcal{O}_{X}\ar[r] &0\
\end{tikzcd}\]
and we claim it is surjective. From the same short exact sequence we see the surjectivity is equivalent to the injecticity of the composition
\[\begin{tikzcd}
&\mathrm{H}^{1}(X,\mathcal{E}_{n}^{\vee})\otimes\mathrm{H}^{1}(X,\mathcal{E}_{0}^{\vee}) \ar[r,"\cup"] &\mathrm{H}^{2}(X,\mathcal{E}_{n}^{\vee}) \ar[r] &\mathrm{H}^{2}(X,\mathcal{E}_{n}^{\vee}/\mathcal{E}_{0}^{\vee}).
\end{tikzcd}\]
It suffices to prove that the images of the cup product $\cup: \mathrm{H}^{1}(X,\mathcal{E}_{n}^{\vee})\otimes \mathrm{H}^{1}(X,\mathcal{O}_{X})\rightarrow \mathrm{H}^{2}(X,\mathcal{E}_{n}^{\vee})$ and the morphism $\mathrm{H}^{2}(X,\mathcal{E}_{0}^{\vee})\rightarrow \mathrm{H}^{2}(X,\mathcal{E}_{n}^{\vee})$ have intersection $0$. However we have a stronger fact (Corollary \ref{null_inter}) saying that the images of the cup product $\cup: \mathrm{H}^{1}(X,\mathcal{E}_{n}^{\vee})\otimes \mathrm{H}^{1}(X,\mathcal{O}_{X})\rightarrow \mathrm{H}^{2}(X,\mathcal{E}_{n}^{\vee})$ and the morphism $\mathrm{H}^{2}(X,\mathcal{E}_{n-1}^{\vee})\rightarrow \mathrm{H}^{2}(X,\mathcal{E}_{n}^{\vee})$ have intersection $0$, and the claim follows from the observation that the morphism  $\mathrm{H}^{2}(X,\mathcal{E}_{0}^{\vee})\rightarrow \mathrm{H}^{2}(X,\mathcal{E}_{n}^{\vee})$ factors through $\mathrm{H}^{2}(X,\mathcal{E}_{n-1}^{\vee})$.  \par
Now we have a surjection \[\mathrm{H}^{1}(\pi_{1}^{N}(X,x),\mathcal{E}_{1,x}^{\vee}/\mathcal{E}^{\vee}_{0,x})\twoheadrightarrow\varinjlim_{n}\mathrm{H}^{1}(\pi_{1}^{N}(X,x),\mathcal{E}_{n,x}^{\vee}/\mathcal{E}^{\vee}_{0,x})\cong \mathrm{H}^{2}(\pi_{1}^{N}(X,x),\mathcal{E}_{0,x}^{\vee}).\] Writing $\mathrm{H}^{1}(\pi_{1}^{N}(X,x),\mathcal{E}_{1,x}^{\vee}/\mathcal{E}^{\vee}_{0,x})$ as $\mathrm{H}^{1}(\pi_{1}^{N}(X,x),\mathcal{E}^{\vee}_{0,x})^{\otimes 2}$, we deduce the surjectivity of the cup product of group cohomology:
\[\cup: \mathrm{H}^{1}(\pi_{1}^{N}(X,x),\mathcal{E}^{\vee}_{0,x})^{\otimes 2} \twoheadrightarrow \mathrm{H}^{2}(\pi_{1}^{N}(X,x),\mathcal{E}^{\vee}_{0,x}),\]
and we conclude from the following comutative diagram:
\[\begin{tikzcd}
&\mathrm{H}^{1}(\pi_{1}^{N}(X,x),\mathcal{E}^{\vee}_{0,x})^{\otimes 2} \ar[r,"\cup",twoheadrightarrow]\ar[d,equal] &\mathrm{H}^{2}(\pi_{1}^{N}(X,x),\mathcal{E}^{\vee}_{0,x}) \ar[d,hook] \\
&\mathrm{H}^{1}(X,\mathcal{E}^{\vee}_{0})^{\otimes 2} \ar[r,"\cup"] &\mathrm{H}^{2}(X,\mathcal{E}^{\vee}_{0}). 
\end{tikzcd}
\]
We explain the the injectivity of the right vertical map from the commutative square above. This follows from the same argument of Example \ref{cohomology_free} and thus that of \cite{BAO2025103646}. The argument of Proposition  2.2  \cite{esnault2006gauss} also leads to a proof of this fact.
\end{proof}
 \bibliographystyle{plain}
 \bibliography{uni_nori}
\end{document}